\documentclass[12pt, twoside, leqno]{article}

\usepackage{amsmath,amsthm}
\usepackage{amssymb}

\usepackage{mathrsfs}

\theoremstyle{plain}
\newtheorem{theorem}{Theorem}[section]
\newtheorem{lemma}[theorem]{Lemma}
\newtheorem{claim}[theorem]{Claim}

\theoremstyle{definition}
\newtheorem{definition}[theorem]{Definition}

\newtheorem{remark}[theorem]{Remark}

\numberwithin{equation}{section}

\newcommand{\R}{\mathbb{R}}
\newcommand{\Q}{\mathbb{Q}}
\newcommand{\N}{\mathbb{N}}

\newcommand{\rank}{\qopname\relax o{rk}}

\newcommand{\len}{\qopname\relax o{length}}
\def\dusb{DUSB}

\begin{document}

\baselineskip=17pt

\title{Classification of bounded Baire class $\xi$ functions}

\author{Viktor Kiss}

\date{}

\maketitle

\renewcommand{\thefootnote}{}



\renewcommand{\thefootnote}{\arabic{footnote}}
\setcounter{footnote}{0}

\begin{abstract}
  Kechris and Louveau showed that each real-valued bounded Baire class 1 
  function defined on a compact metric space can be written as an alternating 
  sum of a decreasing countable transfinite sequence of upper semi-continuous 
  functions. 
  Moreover, the length of the shortest such sequence is essentially the 
  same as the value of certain natural ranks they defined on the Baire 
  class 1 functions. 
  They also introduced the notion of pseudouniform convergence to generate 
  some classes of bounded Baire class 1 functions from others. 
  The main aim of this paper is to generalize their results to 
  Baire class $\xi$ functions. 
  For our proofs to go through, it was essential to first obtain 
  similar results for Baire class 1 functions defined 
  on not necessary compact Polish spaces.
  Using these new classifications of bounded Baire class $\xi$ functions, 
  one can define natural ranks on these classes. 
  We show that these ranks essentially coincide 
  with those defined by Elekes et.~al.~\cite{EKV}.
\end{abstract}

\section{Introduction}

A real-valued function on a completely metrizable 
topological space is of \emph{Baire class 1}, if it is the pointwise limit 
of continuous functions. 
A \emph{rank} on a class of functions is a map assigning an ordinal to each 
member of the class, typically measuring complexity. 

Kechris and Louveau \cite{KL} investigated the properties of 
three natural ranks on Baire class 1 functions on compact metric spaces. 
We will recall their definitions in Section \ref{ss:ranks}. 
They proved, among other things, that these ranks \emph{essentially} coincide 
on bounded 
functions, showing that for a bounded Baire class 1 function $f$ and an
ordinal $1 \le \lambda < \omega_1$, the value of one of these ranks on $f$ 
is at most $\omega^\lambda$ iff the same holds for the other ranks. 
This fact made it possible to define a hierarchy of these functions: 
for a bounded Baire class 1 function $f$, let 
$f \in \mathscr{B}_1^\lambda$, if the value of one (or equivalently, all) 
of these ranks on $f$ is at most $\omega^\lambda$. 

They also proved that every bounded Baire class 1 function $f$ can be 
written as the alternating sum of a decreasing transfinite sequence of 
upper semi-continuous (USC) functions. 
(Recall that a function $g : X \to \R$ is USC if $\{ x \in X : g(x) < c\}$ 
is open in $X$ for every $c \in \R$.) 
Moreover, they showed that the length of the 
shortest such sequence is at most $\omega^\lambda$ if and only if 
$f \in \mathscr{B}_1^\lambda$. Hence, if we consider the length of the 
shortest such sequence as the rank of the function $f$, we 
obtain a new rank on the bounded Baire class $1$ functions that 
coincides essentially with the three ranks investigated by Kechris and 
Louveau. 

They also introduced the notion of pseudouniform convergence, and 
showed that $\mathscr{B}_1^{\lambda + 1}$ contains exactly those bounded 
Baire class 1 functions that can be written as the pseudouniform limit of a 
sequence of functions from $\mathscr{B}_1^\lambda$. For limit $\lambda$, 
they proved that $f \in \mathscr{B}_1^\lambda$ if and only if $f$ 
is the uniform limit of functions 
from $\bigcup_{\eta < \lambda} \mathscr{B}_1^\eta$. 

Elekes, Kiss and Vidny\'anszky \cite{EKV} generalized their results 
concerning ranks to functions defined on general Polish spaces. 
They showed that most of the results proved by Kechris and Louveau remain 
true in this general setting. They defined analogous ranks on the 
Baire class $\xi$ functions. A function is of \emph{Baire class $\xi$} 
for a countable ordinal $\xi > 1$, if it can be written as the 
pointwise limit of functions from smaller classes. Similarly to the Baire 
class 1 case, for a bounded Baire class $\xi$ function $f$ and an
ordinal $1 \le \lambda < \omega_1$, the value of one of these ranks on $f$ 
is at most $\omega^\lambda$ iff the same holds for the other ranks. 
We again denote by $\mathscr{B}_\xi^\lambda$ the set of those bounded 
Baire class $\xi$ functions with value of one (or equivalently, all) 
of these ranks at most $\omega^\lambda$. 

The motivation for investigating ranks on Baire class $\xi$ functions 
came from calculating the so called solvability cardinal of systems of 
difference equations (see \cite{EL}), that are connected to paradoxical 
geometric decompositions (see e.g.~\cite{L1,L}).

This paper is a continuation of the research started in \cite{EKV}. 
The main aim is to generalize the results of Kechris and 
Louveau concerning bounded Baire class $1$ functions to the 
Baire class $\xi$ case. We show that a bounded Baire class $\xi$ 
function $f$ can be written as the alternating sum of a decreasing 
transfinite sequence $(f_\eta)_\eta$ of non-negative 
\emph{semi-Borel class $\xi$} functions 
(i.e.~$\{x : f_\eta(x) < c\} \in \boldsymbol{\Sigma}^0_\xi$ 
for all $c \in \R$ and $\eta$). As in the Baire class 1 case, one 
can define a rank by assigning the length of the shortest such sequence 
to the function $f$. 
We show that this rank is essentially equal to those defined 
in \cite{EKV}. We also show a method of generating 
the family $\mathscr{B}_\xi^{\lambda + 1}$ from 
$\bigcup_{\eta < \lambda} \mathscr{B}_\xi^{\eta + 1}$. 

Our approach is based on topology refinements. Because of this, 
it was essential to obtain the results of Kechris and Louveau 
for Baire class 1 functions defined on general Polish spaces. 
Our proofs build on ideas of Kechris and Louveau, 
however, since they relied on the compactness of the space 
(they used for example the facts that the rank of a characteristic function is 
always a successor ordinal and that a decreasing sequence of 
USC function converging pointwise to $0$ converges uniformly), 
it was necessary reprove their results.

\section{Preliminaries}

Most of the following basic notations and facts can be found in \cite{K}.

Throughout this paper $(X, \tau)$ is an uncountable 
\emph{Polish space}, i.e., a separable and completely metrizable 
topological space. 

For a set $H$ we denote the characteristic function, closure and 
complement of $H$ by $\chi_H$, $\overline{H}$ and $H^c$, 
respectively. 

We use the notation $\boldsymbol{\Sigma}^0_\xi$, 
$\boldsymbol{\Pi}^0_\xi$ and $\boldsymbol{\Delta}^0_\xi$ for 
the $\xi$th \emph{additive, multiplicative and ambiguous classes} 
of the Borel hierarchy, i.e., $\boldsymbol{\Sigma}^0_1 = \tau$, 
$\boldsymbol{\Pi}^0_1 = \{G^c : G \in \tau\}$, 
$$
  \boldsymbol{\Sigma}^0_\xi = 
  \left(\bigcup_{\lambda < \xi} 
  \boldsymbol{\Pi}^0_\lambda\right)_\sigma \quad
  \boldsymbol{\Pi}^0_\xi = 
  \left(\bigcup_{\lambda < \xi} 
  \boldsymbol{\Sigma}^0_\lambda\right)_\delta \text{ and }
  \boldsymbol{\Delta}^0_\xi = 
  \boldsymbol{\Sigma}^0_\xi\cap \boldsymbol{\Pi}^0_\xi, 
$$
where $\mathcal{H}_\sigma = \left\{\bigcup_{n \in \N} H_n : H_n \in 
\mathcal{H}\right\}$ and 
$\mathcal{H}_\delta = \left\{\bigcap_{n \in \N} H_n : H_n \in 
\mathcal{H}\right\}$.

For a function $f : X \to \R$ we write 
$\|f\| = \sup_{x \in X} |f(x)|$, whereas $|f|$ denotes the function 
$x \mapsto |f(x)|$. If $c \in \R$ then we let 
$\{f < c\} = \{x \in X : f(x) < c\}$. We use the notations 
$\{f > c\}, \{f \le c\}$ and $\{f \ge c\}$ similarly. 

We denote the family of real valued functions defined on $X$ that are of Baire class $\xi$ by $\mathcal{B}_\xi$. It is well-known that a function $f$ is of Baire class $\xi$ iff 
$f^{-1}(U) \in \boldsymbol{\Sigma}^0_{\xi + 1}$ for every 
$U \subseteq \R$ open iff $\{f < c\}, \{f > c\} \in 
\boldsymbol{\Sigma}^0_{\xi + 1}$ for every $c \in \R$. 
We use the abbreviation USC for upper semi-continuous functions, i.e., 
a function $f : X \to \R$ is USC if $\{f < c\}$ is open for every 
$c \in \R$. As an analogue, a function $f$ is a \emph{semi-Borel class 
$\xi$} function if $\{f < c\} \in \boldsymbol{\Sigma}^0_\xi$ for 
every $c \in \R$. Note that the pointwise infimum of an arbitrary class of 
non-negative USC functions is USC.

For a countable ordinal $\xi \ge 1$ we denote by $\dusb_\xi$ 
the set of non-negative, bounded, transfinite decreasing sequences of 
semi-Borel class $\xi$ functions $(f_\eta)_{\eta < \lambda}$ with 
$\lambda < \omega_1$ and $f_\eta \to 0$ as $\eta \to \lambda$ for 
limit $\lambda$. 
The \emph{length} of a sequence 
$(f_\eta)_{\eta < \lambda} \in \dusb_\xi$ is 
$\len((f_\eta)_{\eta < \lambda}) = \lambda$. 

If $\tau'$ is a topology on $X$ then we denote the set of Baire class 
$\xi$ functions with respect to $\tau'$ by $\mathcal{B}_\xi(\tau')$. 
Analogously, the notation $\boldsymbol{\Sigma}^0_\xi(\tau')$ stands 
for the $\xi$th additive class of $(X, \tau')$, and similarly for 
$\boldsymbol{\Pi}^0_\xi(\tau')$ and 
$\boldsymbol{\Delta}^0_\xi(\tau')$. 
Moreover, we will use the notation $\dusb_\xi(\tau')$ analogously. 

\subsection{Short introduction to ranks}
\label{ss:ranks}
A \emph{rank} on a class of functions $\mathcal{F}$ is a map assigning 
an ordinal to each $f \in \mathcal{F}$. 
In this section we give the basic definitions about ranks on the Baire class 
$\xi$ functions that we will need. For more on ranks defined on the Baire 
class 1 functions on a compact space see \cite{KL}, and for the 
generalizations for the Baire class $\xi$ functions on Polish spaces see 
\cite{EKV}.

\subsubsection{Derivatives}
The definition of some ranks will use the notion of a 
\emph{derivative operation}. 
A \emph{derivative} on the closed subsets of $X$ is a map 
$D: \boldsymbol{\Pi}^0_1 \to \boldsymbol{\Pi}^0_1$ 
such that $D(A) \subseteq A$ and 
$A \subseteq B \Rightarrow D(A) \subseteq D(B)$ for every 
$A, B \in \boldsymbol{\Pi}^0_1$. 
In the definition below, every derivative operation will satisfy 
these conditions. However, we omit the proofs of these easy facts; 
for a more thorough introduction consult the above references. 

For a derivative $D$ we define the \emph{iterated derivatives} 
of the closed set $F$ as follows: 
\begin{align*}
  D^0(F) &= F, \\
  D^{\theta + 1}(F) &= D(D^{\theta}(F)), \\
  D^{\theta}(F) &= \bigcap_{\eta < \theta} D^{\eta}(F) 
    \text{ if $\theta$ is limit.}
\end{align*}
The \emph{rank} of $D$ is the smallest 
ordinal $\theta$, such that $D^{\theta}(X) = \emptyset$, if such 
ordinal exists, $\omega_1$ otherwise. 
We denote the rank of $D$ by $\rank(D)$.

\subsubsection{Ranks on Baire class 1 functions}

Now we look at ranks on the Baire class 1 functions. The 
\emph{separation rank} has been first introduced by Bourgain \cite{B}.     Let $A$ and $B$ be two subsets of $X$. We associate a derivative 
with them by $D_{A, B}(F) = 
\overline{F \cap A} \cap \overline{F \cap B}$ and denote the rank of 
this derivative by $\alpha(A, B)$. The \emph{separation rank} of a Baire class 1 function $f$ is 
$$
  \alpha(f) = 
  \sup_{\begin{subarray}{c} p < q \\ p, q \in \Q \end{subarray}}
  \alpha(\{f \le p\}, \{f \ge q \}).
$$

The \emph{oscillation rank} was investigated by many authors, 
see e.g.~\cite{HOR}. 
The \emph{oscillation} of a function $f:X \to \R$ at a point 
$x \in X$ restricted to a closed set $F \subseteq X$ is
\begin{equation*}
  \omega(f, x, F) = \inf \left\{ \sup_{x_1, x_2 \in U \cap F} 
  |f(x_1) - f(x_2)|: \text{$U$ open, $x \in U$} \right\}. 
\end{equation*}
For each $\varepsilon > 0$ consider the derivative 
$D_{f, \varepsilon}(F) = 
    \left\{ x \in F : \omega(f, x, F) \ge \varepsilon \right\}.$
The \emph{oscillation rank} of a function $f$ is 
\begin{equation}
  \label{e:beta(f)}
  \beta(f) = \sup_{\varepsilon > 0} \rank(D_{f, \varepsilon}).
\end{equation} 

Next we define the \emph{convergence rank}, see e.g.~Zalcwasser \cite{Z} 
and Gillespie and Hurwitz \cite{GH}. 
Let $(f_n)_{n \in \N}$ be a sequence of real valued functions 
on $X$. The \emph{oscillation} of this sequence at a point $x$ 
restricted to a closed set $F \subseteq X$ is 
\begin{equation*}
  \omega((f_n)_{n \in \N}, x, F) = 
    \inf_{
      \begin{subarray}{c} x \in U \\ \text{$U$ open} 
    \end{subarray}} 
    \inf_{N \in \N} 
    \sup \left\{ |f_m(y) - f_n(y)| : 
    n, m \ge N,\; y \in U \cap F \right\}.     
\end{equation*}
Consider a sequence $(f_n)_{n \in \N}$ of 
functions, and for each $\varepsilon > 0$, let a derivative be defined by 
$D_{(f_n)_{n \in \N}, \varepsilon}(F) = \left\{ x \in F : 
\omega((f_n)_{n \in \N}, x, F) \ge \varepsilon \right\}.$ 
Again, for a sequence $(f_n)_{n \in \N}$ let 
\begin{equation} 
  \label{e:gamma(f_n)}
  \gamma((f_n)_{n \in \N}) = \sup_{\varepsilon > 0} 
  \rank\left(D_{(f_n)_{n \in \N}, \varepsilon}\right).
\end{equation} 
For a Baire class 1 function $f$ let the \emph{convergence rank} 
of $f$ be defined by
\begin{equation}
  \label{e:gamma(f)}
 \gamma(f) = \min \left\{
 \gamma((f_n)_{n \in \N}) : 
   \forall n \text{ $f_n$ is continuous and $f_n \to f$ 
   pointwise} 
 \right\}.
\end{equation}

\subsubsection{Ranks on Baire class $\xi$ functions}

Let $(F_\eta)_{\eta < \lambda}$ be a continuous, (i.e, 
for a limit ordinal $\theta < \lambda$, 
$\bigcap_{\eta < \theta} F_\eta = F_\theta$) decreasing 
sequence of $\boldsymbol{\Pi}^0_\xi$ sets for some $\lambda < \omega_1$ 
with $F_0 = X$ and $\bigcap_{\eta < \lambda} F_\eta = \emptyset$ if 
$\lambda$ is limit. We say that the sets $A$ and $B$ can be separated by 
the transfinite difference of this sequence if 
$$
  A \subseteq \bigcup_{\begin{subarray}{c}
  \eta < \lambda \\ \text{$\eta$ even}
  \end{subarray}} F_\eta \setminus F_{\eta + 1} 
  \subseteq B^c,
$$
where $F_\eta = \emptyset$ if $\eta \ge \lambda$. 
By $\alpha_\xi(A, B)$ we denote the length of the shortest such sequence if 
there is any,
otherwise we let $\alpha_\xi(A, B) = \omega_1$. 
We define the \emph{modified separation rank} of a 
Baire class $\xi$ function $f$ as 
\begin{equation*}
  \alpha_\xi(f) = 
    \sup_{\begin{subarray}{c} p < q \\ p, q \in \Q \end{subarray}}
    \alpha_\xi(\{f \le p\}, \{f \ge q \}).
\end{equation*}

Now we introduce one of the methods used in \cite{EKV} to construct 
ranks on the Baire class $\xi$ functions from existing ranks on the 
Baire class 1 functions.

Let $f$ be of Baire class $\xi$. Let 
\begin{equation}
  \label{e:T_(f, xi) def}
  T_{f, \xi} = \{\tau':\tau' \supseteq \tau \text{ Polish}, 
  \tau' \subseteq \boldsymbol{\Sigma}^0_\xi(\tau), 
  f \in \mathcal{B}_1(\tau')\}.
\end{equation}
Let $\rho$ be a rank on the Baire class 1 functions and let  
\begin{equation*}
  \rho_\xi^*(f) = \min_{\tau' \in T_{f, \xi}} \rho_{\tau'}(f), 
\end{equation*}
where $\rho_{\tau'}(f)$ is just the $\rho$ rank of $f$ in the 
topology $\tau'$. This method yields the rank $\rho^*_\xi$ on the 
Baire class $\xi$ functions.

We use the notation 
$$
  \mathscr{B}_\xi^\lambda = 
  \{f \in \mathcal{B}_\xi : 
  \text{$f$ is bounded and } \alpha_\xi(f) \le \omega^\lambda\}.
$$
We also use the notation $\mathscr{B}_\xi^\lambda(\tau')$ for the 
corresponding class with respect to the topology $\tau'$ on $X$. 
Note that by \cite[3.14]{EKV} and \cite[3.35]{EKV} for a bounded 
Baire class 1 function $f$ we have 
$f \in \mathscr{B}_1^\lambda \Leftrightarrow \alpha(f) \le \omega^\lambda 
\Leftrightarrow \beta(f) \le \omega^\lambda \Leftrightarrow 
\gamma(f) \le \omega^\lambda$ and by \cite[5.7]{EKV}, for a bounded 
function $f \in \mathcal{B}_\xi$ we have 
$f \in \mathscr{B}_\xi^\lambda \Leftrightarrow \alpha_\xi^*(f) \le \omega^\lambda 
\Leftrightarrow \beta_\xi^*(f) \le \omega^\lambda \Leftrightarrow 
\gamma_\xi^*(f) \le \omega^\lambda$. 

\begin{remark}
  \label{r:bounded classes}
  For a function $f$, $f \in \mathscr{B}_\xi^\lambda$ if and only if 
  there exists a topology $\tau' \in T_{f, \xi}$ such that 
  $f \in \mathscr{B}_1^\lambda(\tau')$. This can be easily seen 
  as $f \in \mathscr{B}_\xi^\lambda \Leftrightarrow 
  \alpha^*_\xi(f) \le \omega^\lambda \Leftrightarrow 
  \exists \tau' \in T_{f, \xi} (\alpha_{\tau'}(f) \le \omega^\lambda) 
  \Leftrightarrow \exists \tau' \in T_{f, \xi} 
  (f \in \mathscr{B}_1^\lambda(\tau'))$.
\end{remark}

Now we prove three lemmas about these ranks that will be useful later on.

\begin{lemma}
  \label{l:alpha = beta for char}
  For a characteristic function $\chi_A \in \mathcal{B}_1$, 
  $\alpha(f) = \beta(f)$. 
\end{lemma}
\begin{proof}
  It is enough to prove that for every  $\varepsilon < 1$ and 
  $F \subseteq X$ closed, we have $D_{\{\chi_A \le 0\},\{\chi_A \ge 1\}}(F) = 
  D_{\chi_A, \varepsilon}(F)$. 
  Let $x \in X$ then $x \in D_{\chi_A, \varepsilon}(F) \Leftrightarrow 
  \omega(f, x, F) \ge \varepsilon \Leftrightarrow 
  \left(x \in U \text{ is open } \Rightarrow \exists y,z \in U \cap F 
  (y \in A \land z \not \in A) \right)\Leftrightarrow 
  x \in \overline{F \cap A} \cap \overline{F \cap A^c} \Leftrightarrow 
  x \in D_{\{\chi_A \le 0\},\{\chi_A \ge 1\}}(F)$. 
\end{proof}
\begin{lemma}
  \label{l:gamma is additive}
  Let $(f_n)_{n \in \N}, (g_n)_{n \in \N}$ be two sequences 
  of functions such that  
  $\gamma((f_n)_{n \in \N}), \gamma((g_n)_{n \in \N}) \le 
  \omega^\lambda$ for some $\lambda < \omega_1$. Then 
  $\gamma((f_n + g_n)_{n \in \N}) \le \omega^\lambda$.
\end{lemma}
\begin{proof}
  By Theorem 3.29 in \cite{EKV}, the rank $\gamma$ defined 
  on $\mathcal{B}_1$ satisfies $\gamma(f + g) \le \omega^\lambda$ whenever 
  $\gamma(f), \gamma(g) \le \omega^\lambda$. But they 
  actually prove the statement of this lemma and derive the theorem 
  from this fact.
\end{proof}
\begin{lemma}
  \label{l:beta(g f) <= beta(f)}
  If $f : X \to \R$ is a function and $g : \R \to \R$ is a Lipschitz 
  map then $\beta(g \circ f) \le \beta(f)$. 
\end{lemma}
\begin{proof}
  Let the Lipschitz constant of $g$ be $c$. 
  Then one can easily see that 
  $\omega(g \circ f, x, F) \le c \cdot \omega(f, x, F)$ for every 
  $x \in X$ and $F \subseteq X$ closed, hence 
  $\rank(D_{g \circ f, c \cdot \varepsilon}) \le 
  \rank(D_{f, \varepsilon})$, 
  showing that $\beta(g \circ f) \le \beta(f)$. 
\end{proof}

\section{The alternating sums of semi-Borel class $\xi$ functions}

Now we define the notion of an alternating sum of a transfinite 
sequence of semi-Borel class $\xi$ functions. It is the generalization of the alternating sum of USC functions defined by A.~S.~Kechris and A.~Louveau in \cite{KL}.

\begin{definition}
  Let $\lambda$ be a countable ordinal and let 
  $(f_\eta)_{\eta < \lambda} \in \dusb_\xi$. 
  The function 
  $\sideset{}{^*}\sum_{\eta < \theta} (-1)^\eta f_\eta$ is defined 
  inductively on $\theta \le \lambda$, by 
  \begin{equation*}
    \sideset{}{^*}\sum_{\eta < \theta + 1} (-1)^\eta f_\eta = 
    \sideset{}{^*}\sum_{\eta < \theta} (-1)^\eta f_\eta + 
      (-1)^\theta f_\theta,
  \end{equation*}
  where $(-1)^\theta = 1$ if $\theta$ is even and $-1$ if $\theta$ is 
  odd, and for limit $\theta \le \lambda$
  \begin{equation*}
    \sideset{}{^*}\sum_{\eta < \theta} (-1)^\eta f_\eta = 
    \sup\left\{\sideset{}{^*}\sum_{\eta < \zeta} (-1)^\eta f_\eta : 
    \text{$\zeta$ is even, } \zeta < \theta \right\}.
  \end{equation*}
  
  For a function $f$ if  
  $(f_\eta)_{\eta < \lambda} \in \dusb_\xi$ is a sequence with 
  $f = c + \sideset{}{^*}\sum_{\eta < \lambda} (-1)^\eta f_\eta$ for 
  some $c \in \R$,  then we 
  say that $f$ is the sum of a constant and the alternating 
  sequence $(f_\eta)_{\eta < \lambda}$ of length $\lambda$. 
  We use the notation 
  \begin{equation*}
    \len_\xi(f) = \inf\bigg\{\lambda : \exists
    (f_\eta)_{\eta < \lambda} \in \dusb_\xi, c \in \R \bigg(
    f = c + \sideset{}{^*}\sum_{\eta < \lambda} (-1)^\eta 
    f_\eta\bigg)\bigg\}, 
  \end{equation*}
  where we define $\len_\xi(f)$ to be $\omega_1$ if $f$ is not the sum 
  of a constant and an alternating sequence from $\dusb_\xi$.
\end{definition}

\begin{remark}
  \label{r:sup=limit}
  It is easy to prove by transfinite induction that for even ordinals 
  $\theta_1 \le \theta_2$ we have 
  \begin{equation}
    \label{e:monotony of alternating sum}
    \sideset{}{^*}\sum_{\eta < \theta_1} (-1)^\eta f_\eta \le 
    \sideset{}{^*}\sum_{\eta < \theta_2} (-1)^\eta f_\eta. 
  \end{equation}
  From this fact, for limit $\theta$ if $\theta_n \to \theta$, 
  $\theta_n < \theta$ even then 
  \begin{equation}
    \label{e:sup=limit}
    \sideset{}{^*}\sum_{\eta < \theta} (-1)^\eta f_\eta = 
    \lim_{n \to \infty}
    \sideset{}{^*}\sum_{\eta < \theta_n} (-1)^\eta f_\eta.
  \end{equation}
  We will use this fact to calculate 
  $\sideset{}{^*}\sum_{\eta < \theta} (-1)^\eta f_\eta$.
\end{remark}

\begin{remark}
  \label{r:first summand dominates}
  Let $(f_\eta)_{\eta < \lambda} \in \dusb_\xi$ and $\theta \le \lambda$ 
  with $\theta$ even. 
  We show by transfinite induction on $\zeta$ that for every 
  $\theta \le \zeta \le \lambda$ even, we have 
  \begin{equation}
    \label{e:alternating sum general}
    0 \le \sideset{}{^*}\sum_{\eta < \zeta} (-1)^\eta f_\eta - 
      \sideset{}{^*}\sum_{\eta < \theta} (-1)^\eta f_\eta \le 
    f_\theta - f_\zeta.
  \end{equation}
  For $\zeta + 2$ we have 
  \begin{equation*}
  \begin{split}
    0 \le \sideset{}{^*}\sum_{\eta < \zeta} (-1)^\eta f_\eta - 
      \sideset{}{^*}\sum_{\eta < \theta} (-1)^\eta f_\eta \le \\
    \sideset{}{^*}\sum_{\eta < \zeta} (-1)^\eta f_\eta + 
      f_\zeta - f_{\zeta + 1} - 
      \sideset{}{^*}\sum_{\eta < \theta} (-1)^\eta f_\eta \le \\
    f_\theta - f_\zeta + f_\zeta - f_{\zeta + 1} \le 
    f_\theta - f_{\zeta + 2}, 
  \end{split}
  \end{equation*}
  where the expression in the middle equals to 
  $$
    \sideset{}{^*}\sum_{\eta < \zeta + 2} (-1)^\eta f_\eta - 
      \sideset{}{^*}\sum_{\eta < \theta} (-1)^\eta f_\eta,
  $$ 
  proving the successor case. For limit $\zeta$, 
  \eqref{e:alternating sum general} is an easy consequence of 
  \eqref{e:sup=limit} and the monotonicity of the sequence 
  $(f_\eta)_{\eta < \lambda}$. 
  
  Now let $f = \sideset{}{^*}\sum_{\eta < \lambda} (-1)^\eta 
  f_\eta$. Since the alternating sum of a sequence does not change if we 
  append $0$ functions to it, we can suppose that $\lambda$ is 
  even. Hence we can substitute $\zeta = \lambda$ to get 
  \begin{equation}
    \label{e:alternating sum spec}
    0 \le f - \sideset{}{^*}\sum_{\eta < \theta} (-1)^\eta f_\eta \le 
    f_\theta,
  \end{equation}
  in particular, 
  \begin{equation}
    \label{e:f <= f_0}
    0 \le f \le f_0.
  \end{equation}
\end{remark}

\begin{theorem}
  \label{t:Baire 1 alternating}
  Let $f$ be a bounded Baire class 1 function. 
  Then $f \in \mathscr{B}_1^\lambda$ if and only if 
  $\len_1(f) \le \omega^\lambda$. 
\end{theorem}
\begin{remark}
  A straightforward consequence of this theorem is that every 
  bounded Baire class 1 function can be written as the sum of a constant 
  and an alternating sequence from $\dusb_1$ (as $\alpha_1(f) < \omega_1$ 
  for every Baire class 1 function $f$, see \cite[3.15]{EKV}). 
  For the other direction, 
  that if $f$ can be written in this form then $f$ is a bounded 
  Baire class 1 function, see \cite{EV}. 
\end{remark}
\begin{proof}[Proof of Theorem \ref{t:Baire 1 alternating}]
  It is easy to see that it is enough to prove the theorem for 
  non-negative functions, since for any constant $c$, 
  $f \in \mathscr{B}_1^\lambda \Leftrightarrow f + c \in 
  \mathscr{B}_1^\lambda$ and $\len_1(f + c) = \len_1(f)$. 
  We first show that if $f \in \mathscr{B}_1^\lambda$ then 
  $\len_1(f) \le \omega^\lambda$. 

  Let $f \in \mathscr{B}_1^{\lambda}$ be a characteristic 
  function, i.e., $f = \chi_A$ for some $A \subseteq X$. Using the 
  definition of $\mathscr{B}_1^{\lambda}$, we can separate 
  $\{f \ge 1\} = A$ and $\{f \le 0\} = A^c$ with an appropriate sequence, 
  hence $A$ can be written as 
  $$
    A = \bigcup_{\begin{subarray}{c}
      \eta < \omega^{\lambda} \\ \text{$\eta$ is even}
    \end{subarray}} F_\eta \setminus F_{\eta + 1},
  $$
  where $(F_\eta)_{\eta < \omega^{\lambda}}$ is a decreasing, 
  continuous sequence of closed sets with $F_0 = X$ and 
  $\bigcap_{\eta < \omega^{\lambda}} F_\eta = \emptyset$.

  Now let $f_\eta = \chi_{F_\eta}$. It is easy 
  to see that $(f_\eta)_{\eta < \omega^\lambda}$ is a decreasing sequence 
  of non-negative, bounded USC functions with $f_\eta \to 0$ as 
  $\eta \to \omega^\lambda$. From this 
  $(f_\eta)_{\eta < \omega^\lambda} \in \dusb_1$, hence to prove that 
  $\len_1(f) \le \omega^\lambda$, it is enough to 
  prove that 
  $f = \sideset{}{^*}\sum_{\eta < \omega^\lambda} (-1)^\eta f_\eta$. 
  We do this by proving that for every 
  $\theta \le \omega^\lambda$ even we have 
  \begin{equation*}
    \sideset{}{^*}\sum_{\eta < \theta} (-1)^\eta f_\eta = 
    \chi_{\bigcup_{\begin{subarray}{c} \eta < \theta \\ 
    \text{$\eta$ even}\end{subarray}} F_\eta \setminus F_{\eta + 1}}.
  \end{equation*}
  For $\theta = 0$ this is obvious. Suppose this holds for $\theta$ then 
  \begin{equation*}
  \begin{split}
    \sideset{}{^*}\sum_{\eta < \theta + 2} (-1)^\eta f_\eta = 
    \sideset{}{^*}\sum_{\eta < \theta} (-1)^\eta f_\eta + 
      f_\theta - f_{\theta + 1} = \\
    \chi_{\bigcup_{\begin{subarray}{c} \eta < \theta \\   
      \text{$\eta$ even}\end{subarray}} F_\eta 
        \setminus F_{\eta + 1}} + 
      \chi_{F_\theta} - \chi_{F_{\theta + 1}} = 
    \chi_{\bigcup_{\begin{subarray}{c} \eta < \theta + 2 \\   
      \text{$\eta$ even}\end{subarray}} F_\eta \setminus F_{\eta + 1}}.    
  \end{split}
  \end{equation*}
  For limit $\theta$ let $\theta_n \to \theta$, $\theta_n < \theta$ even 
  then 
  \begin{equation*}
  \begin{split}
    \sideset{}{^*}\sum_{\eta < \theta} (-1)^\eta f_\eta = 
    \lim_{n \to \infty}
      \sideset{}{^*}\sum_{\eta < \theta_n} (-1)^\eta f_\eta =
    \lim_{n \to \infty}
      \chi_{\bigcup_{\begin{subarray}{c} \eta < \theta_n \\
      \text{$\eta$ even}\end{subarray}} F_\eta \setminus F_{\eta + 1}}
    = \chi_{\bigcup_{\begin{subarray}{c} \eta < \theta \\
      \text{$\eta$ even}\end{subarray}} F_\eta \setminus F_{\eta + 1}},
  \end{split}
  \end{equation*}
  proving $\len_1(f) \le \omega^\lambda$ for the characteristic 
  function $f \in \mathscr{B}_1^\lambda$. 
  
  Now let $f \in \mathscr{B}_1^\lambda$ be a non-negative step 
  function, that is, a linear combination 
  of characteristic functions. Such a function can be written as 
  $f = \sum_{i = 1}^n c_i \chi_{A_i}$ where the $c_i$'s 
  are distinct, non-negative real numbers and the $A_i$'s 
  form a partition of $X$ with $A_i \in \boldsymbol{\Delta}^0_2$ for 
  each $i$. By the above statement, each $\chi_{A_i}$ can be written as 
  $\chi_{A_i} = \sideset{}{^*}\sum_{\eta < \omega^\lambda} (-1)^\eta f^i_\eta$, 
  where $(f^i_\eta)_{\eta < \omega^\lambda} \in \dusb_1$, since 
  $\alpha_1(\chi_{A_i}) \le \omega^\lambda$ (see \cite[3.38]{EKV} and 
  \cite[3.14]{EKV}). 
  Now let $f_\eta = \sum_{i = 1}^n c_i \cdot f^i_\eta$. 
  It is easy to see that 
  $(f_\eta)_{\eta < \omega^\lambda} \in \dusb_1$ and 
  $f = \sideset{}{^*}\sum_{\eta < \omega^\lambda} (-1)^\eta f_\eta$, 
  showing that $\len_1(f) \le \omega^\lambda$ for step functions 
  $f \in \mathscr{B}_1^\lambda$. 
  Moreover, this construction shows that the $f_\eta$'s can be chosen 
  in such a way that 
  \begin{equation}
    \label{e:||f_eta|| <= ||f|| for step functions}
    \|f_\eta\| \le \|f\|.
  \end{equation}
  
  Now we turn to the case of arbitrary non-negative bounded functions. 
  \begin{lemma}
    \label{l:sum of step functions}
    If $f \in \mathscr{B}_1^\lambda$ then there exists a sequence 
    $(g^k)_{k \in \N}$ of non-negative step functions 
    $g^k \in \mathscr{B}_1^\lambda$ such that 
    $\inf f + \sum_{k \in \N} g^k = f$ and 
    $\|g^k\| \le \frac{1}{2^k}$ for $k \ge 1$. 
  \end{lemma}
  \begin{proof}
    It is enough to show that there exists such a sequence with 
    $\sum_{k \in \N} g^k = f$ for a non-negative function $f$, 
    since $f - \inf f \in \mathscr{B}_1^\lambda$ is always 
    non-negative. 
    
    So let $f \in \mathscr{B}_1^\lambda$ be non-negative. 
    Then there exists a sequence of step functions 
    $(f^k)_{k \in \N}$ converging uniformly to $f$ with 
    $f^k \in \mathscr{B}_1^\lambda$ for every $k \in \N$ 
    (see \cite[3.40]{EKV}). By taking a 
    subsequence, we can suppose that 
    $\|f^k - f\| \le \frac{1}{2^{k + 5}}$. 
    By substituting $f^k$ with $\max\{f^k - \frac{1}{2^{k + 3}}, 0\}$, 
    we can suppose moreover that $(f^k)_{k \in \N}$ is an 
    increasing sequence of non-negative functions now satisfying 
    $\|f^k - f\| \le \frac{1}{2^{k + 2}}$, and 
    using Lemma \ref{l:beta(g f) <= beta(f)}, we still have 
    $f^k \in \mathscr{B}_1^\lambda$.
    
    Let $g^0 = f^0$ and for $k \ge 1$ let $g^k = f^k - f^{k - 1}$. 
    Then $g^k \ge 0$, $\|g^k\| \le \frac{1}{2^k}$ for $k \ge 1$ and 
    $\sum_{k \in \N} g^k = f$. 
    By \cite[3.29]{EKV}, $g^k \in \mathscr{B}_1^\lambda$, proving 
    the lemma. 
  \end{proof}
  
  Now let $(g^k)_{k \in \N}$ be the sequence given by the lemma and 
  substitute $g^0$ with $g^0 + \inf f$. 
  Then $g^0$ remains non-negative and now 
  $\sum_{k \in \N} g^k = f$. Since $g^k \in \mathscr{B}_1^\lambda$ 
  is a step function for each $k$, we can write 
  $$
    g^k = \sideset{}{^*}\sum_{\eta < \omega^\lambda} (-1)^\eta g^k_\eta, 
  $$ 
  where $(g^k_\eta)_{\eta < \omega^\lambda} \in \dusb_1$ and each 
  $g^k_\eta$ is chosen to satisfy 
  \eqref{e:||f_eta|| <= ||f|| for step functions}, 
  hence $\|g^k_\eta\| \le \|g^k\|$. 
  
  For $\eta < \omega^\lambda$ let 
  $$
    f_\eta = \sum_{k \in \N} g^k_\eta.
  $$ 
  We claim that $(f_\eta)_{\eta < \omega^\lambda} \in \dusb_1$ and 
  $$
    f = \sideset{}{^*}\sum_{\eta < \omega^\lambda} (-1)^\eta f_\eta.
  $$ 
  It is enough to show these claims to finish the proof of the implication  
  $f \in \mathscr{B}_1^\lambda \Rightarrow \len_1(f) \le \omega^\lambda$. 
  
  Since $\|g^k_\eta\| \le \frac{1}{2^k}$ for $k \ge 1$, 
  $\|g^0_\eta\| \le \|g^0\|$ and $(g^k_\eta)_{\eta < \omega^\lambda} \in 
  \dusb_1$, the sequence $(f_\eta)_{\eta < \omega^\lambda}$ is 
  a non-negative, bounded, decreasing sequence of USC functions, as 
  the finite sum and uniform limit of USC functions is USC. 
  
  Now we show that $f_\eta \to 0$ as $\eta \to \omega^\lambda$. 
  Let $x \in X$ and $\varepsilon > 0$ be fixed. There exists a 
  $k_0$ with $\sum_{k \ge k_0} g^k_\eta(x) \le 
  \sum_{k \ge k_0} \frac{1}{2^k} < \frac{\varepsilon}{2}$. 
  For this $k_0$, we can find an ordinal $\lambda_0 < \omega^\lambda$ 
  such that for every $\lambda_0 \le \eta < \omega^\lambda$ and 
  $k < k_0$, $g^k_\eta(x) < \frac{\varepsilon}{2k_0}$, 
  since $g^k_\eta \to 0$ as $\eta \to \omega^\lambda$ for each $k$. 
  Hence for every $\lambda_0 \le \eta < \omega^\lambda$ we have 
  $f_\eta(x) \le \varepsilon$, showing that $f_\eta \to 0$ as 
  $\eta \to \omega^\lambda$, thus proving 
  $(f_\eta)_{\eta < \omega^\lambda} \in \dusb_1$. 
  
  To show that 
  $f = \sideset{}{^*}\sum_{\eta < \omega^\lambda} (-1)^\eta f_\eta$, 
  we prove by transfinite induction that for every 
  $\theta \le \omega^\lambda$, 
  \begin{equation*}
    \sideset{}{^*}\sum_{\eta < \theta} (-1)^\eta f_\eta = 
    \sum_{k \in \N} 
    \sideset{}{^*}\sum_{\eta < \theta} (-1)^\eta g^k_\eta.
  \end{equation*}
  Suppose this holds for $\theta$, then
  \begin{equation*}
  \begin{split}
    \sideset{}{^*}\sum_{\eta < \theta + 1} (-1)^\eta f_\eta = 
    \sideset{}{^*}\sum_{\eta < \theta} (-1)^\eta f_\eta + 
      (-1)^\theta f_\theta = \\
    \sum_{k \in \N} 
      \sideset{}{^*}\sum_{\eta < \theta} (-1)^\eta g^k_\eta + 
      \sum_{k \in \N} (-1)^\theta g^k_\theta = 
    \sum_{k \in \N} 
      \; \sideset{}{^*}\sum_{\eta < \theta + 1} (-1)^\eta g^k_\eta.
  \end{split}
  \end{equation*}
  And for limit $\theta$ let $\theta_n \to \theta$, 
  $\theta_n < \theta$ even then 
  \begin{equation*}
  \begin{split}
    \sideset{}{^*}\sum_{\eta < \theta} (-1)^\eta f_\eta = 
    \lim_{n \to \infty}
      \sideset{}{^*}\sum_{\eta < \theta_n} (-1)^\eta f_\eta = 
    \lim_{n \to \infty}
    \sum_{k \in \N} 
      \sideset{}{^*}\sum_{\eta < \theta_n} (-1)^\eta g^k_\eta = \\
    \sum_{k \in \N} \lim_{n \to \infty}
      \sideset{}{^*}\sum_{\eta < \theta_n} (-1)^\eta g^k_\eta = 
    \sum_{k \in \N} 
      \sideset{}{^*}\sum_{\eta < \theta} (-1)^\eta g^k_\eta, 
  \end{split}
  \end{equation*}
  where we used the dominated convergence theorem to interchange the 
  operators $\lim$ and $\sum$: for a fixed $x \in X$ let 
  $$
    h_n(k) = \sideset{}{^*}\sum_{\eta < \theta_n} (-1)^\eta 
    g^k_\eta(x) \text{\;\; and \;\;} 
    h(k) = \sideset{}{^*}\sum_{\eta < \theta} (-1)^\eta g^k_\eta(x).
  $$
  Then $h_n(k)$ converges to $h(k)$ for every $k$,
  and for every $n \in \N$ by \eqref{e:monotony of alternating sum} and 
  \eqref{e:f <= f_0} we have $|h_n(k)| \le H(k)$, where 
  $H(k) = \|g^k_0\|$. The function $H(k)$ is summable, since 
  $H(k) \le \frac{1}{2^k}$ for $k \ge 1$, hence we can apply 
  the dominated convergence theorem to get that 
  $\lim_{n \to \infty} \sum_{k \in \N} h_n(k) = \sum_{k \in \N} 
  \lim_{n \to \infty} h_n(k)$.
  This finishes the proof of $\len_1(f) \le \omega^\lambda$ for a 
  function $f \in \mathscr{B}_1^\lambda$. 
  
  Now we prove the following two statements by transfinite induction 
  on $\lambda$: 
  \begin{equation}
    \label{e:tending to zero}
    \text{if } 
      f = \sideset{}{^*}\sum_{\eta < \omega^\lambda} (-1)^\eta f_\eta 
      \text{ with } (f_\eta)_{\eta < \omega^\lambda} \in \dusb_1 
      \text{ then } f \in \mathscr{B}_1^\lambda
  \end{equation}
  \begin{equation}
    \label{e:not tending to zero}
    \text{if } 
      f = \sideset{}{^*}\sum_{\eta < \omega^\lambda} (-1)^\eta f_\eta 
      \text{ with } (f_\eta)_{\eta < \omega^\lambda} \in \dusb_1'
      \text{ then } f \in \mathscr{B}_1^{\lambda + 1}, 
  \end{equation}
  where $\dusb_1'$ consists of decreasing, transfinite sequences of 
  bounded, non-negative USC functions of countable length, i.e., we 
  do not assume that $f_\eta \to 0$ as $\eta \to \omega^\lambda$ for 
  the sequence $(f_\eta)_{\eta < \omega^\lambda} \in \dusb_1'$. 
  It is easy to see that \eqref{e:tending to zero} yields the second 
  part of the theorem, hence it is enough to prove these two statements.
  
  First we prove \eqref{e:tending to zero} for $\lambda + 1$ while 
  supposing \eqref{e:tending to zero} and \eqref{e:not tending to zero} 
  for $\lambda$. So let 
  $f = \sideset{}{^*}\sum_{\eta < \omega^{\lambda + 1}} (-1)^\eta 
  f_\eta$, where $(f_\eta)_{\eta < \omega^{\lambda + 1}} \in \dusb_1$. 
  Let $f^k = \sideset{}{^*}\sum_{\eta < \omega^{\lambda}\cdot k} 
  (-1)^\eta f_\eta$, by \eqref{e:sup=limit} we have $f^k \to f$. 
  \begin{claim}
    $\beta(f^k) \le \omega^{\lambda + 1}$.
  \end{claim}
  \begin{proof}
    We prove this by induction on $k$. For $k = 1$ this is 
    \eqref{e:not tending to zero} for $\lambda$ as the sequence 
    $(f_\eta)_{\eta < \omega^\lambda}$ is in $\dusb_1'$. 
    For $k + 1$ we have 
    $f^{k + 1} = f^k + g^k$, where $g^k = f^{k + 1} - f^k$. 
    We have $g^k = \sideset{}{^*}\sum_{\eta < \omega^\lambda} 
    (-1)^\eta f'_\eta$, where 
    $f'_\eta = f_{\omega^\lambda \cdot k + \eta}$ with 
    $(f'_\eta)_{\eta < \omega^\lambda} \in \dusb'_1$. 
    Now using \eqref{e:not tending to zero} for $g^k$ we have 
    $g^k \in \mathscr{B}_1^{\lambda + 1}$, hence 
    $f^{k + 1} = f^k + g^k \in \mathscr{B}_1^{\lambda + 1}$ 
    using \cite[3.29]{EKV} to show that 
    $\beta(f^k), \beta(g^k) \le \omega^{\lambda + 1}$ implies 
    $\beta(f^{k + 1}) \le \omega^{\lambda + 1}$.
  \end{proof}
  
  Now we prove $f \in \mathscr{B}_1^{\lambda + 1}$ by showing 
  that $\beta(f) \le \omega^{\lambda + 1}$. 
  Let $x \in X$, it is enough to prove that 
  $x \not \in D_{f, \varepsilon}^{\omega^{\lambda + 1}}(X)$ for every 
  $\varepsilon > 0$. 
  By \eqref{e:alternating sum spec} we have 
  $0 \le f - f^k \le f_{\omega^\lambda \cdot k}$, hence there exists a 
  $k$ such that $|f(x) - f^k(x)| \le f_{\omega^\lambda \cdot k}(x) \le 
  \frac{\varepsilon}{5}$. Since $f_{\omega^\lambda \cdot k}$ is USC, 
  we have an open set $x \in U$ such that $|f(y) - f^k(y)| \le 
  f_{\omega^\lambda \cdot k}(y) \le \frac{\varepsilon}{4}$ for every 
  $y \in U$. Now we need the following lemma. 
  \begin{lemma}
    \label{l:derivatives of functions}
    If $f$ and $g$ are two Baire class 1 functions, $U$ is open and 
    $F$ is closed with 
    $|f(y) - g(y)| \le \frac{\varepsilon}{4}$ for every $y \in F \cap U$ 
    then for every $\eta < \omega_1$,
    $$
      D_{f, \varepsilon}^\eta(F) \cap U \subseteq 
      D_{g, \frac{\varepsilon}{4}}^\eta(F) \cap U.
    $$
  \end{lemma}
  \begin{proof}
    The proof is by transfinite induction on $\eta$. For $\eta = 0$ this 
    is obvious from the definition of the derivative. 
    Let $x \in \left(D_{g, \frac{\varepsilon}{4}}^\eta(F) \cap U \right)
    \setminus D_{g, \frac{\varepsilon}{4}}^{\eta + 1}(F)$, 
    we need to show that $x \not \in D_{f, \varepsilon}^{\eta + 1}(F)$. 
    There is an open neighborhood $x \in V \subseteq U$ such that 
    $|g(y) - g(z)| < \frac{\varepsilon}{4}$ for every 
    $y, z \in D_{g, \frac{\varepsilon}{4}}^\eta(F) \cap V$. Then 
    $|f(y) - f(z)| < \frac{3}{4}\varepsilon$, 
    for every $y,z \in D_{g, \frac{\varepsilon}{4}}^\eta(F) \cap V$. 
    By the induction hypothesis $D_{f, \varepsilon}^\eta(F) \cap V 
    \subseteq D_{g, \frac{\varepsilon}{4}}^\eta(F) \cap V$, 
    hence this holds for every 
    $y,z \in D_{f, \varepsilon}^\eta(F) \cap V$, 
    thus $x \not \in D_{f, \varepsilon}^{\eta + 1}(F)$. This shows 
    the successor case, and for limit $\eta$ the lemma is an easy 
    consequence of the definition of the derivative. 
  \end{proof}
  Applying the lemma with $g = f^k$, $F = X$ and 
  $\eta = \omega^{\lambda + 1}$, 
  we get that $D_{f, \varepsilon}^{\omega^{\lambda + 1}}(X) \cap U 
  \subseteq D_{f^k, \frac{\varepsilon}{4}}^{\omega^{\lambda + 1}} (X)
  \cap U = \emptyset$, since $\beta(f^k) \le \omega^{\lambda + 1}$. 
  This shows that 
  $x \not \in D_{f, \varepsilon}^{\omega^{\lambda + 1}}(X)$, proving 
  \eqref{e:tending to zero} for the successor case. 
  
  The proof of \eqref{e:tending to zero} for the limit case is similar. 
  Let $\lambda$ be a limit ordinal and let $\lambda_k \to \lambda$, 
  $\lambda_k < \lambda$. 
  Let 
  $$
    f^k = \sideset{}{^*}\sum_{\eta < \omega^{\lambda_k}} (-1)^\eta 
  f_\eta.
  $$ By \eqref{e:not tending to zero} for $\lambda_k < \lambda$ 
  we have $f^k \in \mathscr{B}_1^{\lambda_k + 1} \subseteq 
  \mathscr{B}_1^{\lambda}$. 
  Again by \eqref{e:alternating sum spec}, 
  $0 \le f - f^k \le f_{\omega^{\lambda_k}}$, and using that 
  $f_\eta \to 0$ and $f_\eta$ is USC, for a fixed $x \in X$ we get 
  a neighborhood $x \in U$ and a $k$ such that 
  $|f(y) - f^k(y)| \le \frac{\varepsilon}{4}$ for every $y \in U$. 
  The application of Lemma \ref{l:derivatives of functions} yields 
  $D_{f, \varepsilon}^{\omega^\lambda}(X) \cap U 
  \subseteq D_{f^k, \frac{\varepsilon}{4}}^{\omega^\lambda}(X) 
  \cap U = \emptyset$, hence 
  $x \not \in D_{f, \varepsilon}^{\omega^\lambda}(X)$. As we started 
  with an arbitrary $x \in X$, this shows 
  $D_{f, \varepsilon}^{\omega^\lambda}(X) = \emptyset$, thus 
  $\beta(f) \le \omega^\lambda$, proving $f \in \mathscr{B}_1^\lambda$. 
  
  It remains to prove \eqref{e:not tending to zero}. Now we can use 
  \eqref{e:tending to zero} for $\lambda$ as we proved it using 
  \eqref{e:not tending to zero} only for smaller ordinals. 
  Let $(f_\eta)_{\eta < \omega^\lambda} \in \dusb_1'$ and 
  $\lambda_k \to \omega^\lambda$, $\lambda_k < \omega^\lambda$ even. 
  Let 
  $$
    f = \sideset{}{^*}\sum_{\eta < \omega^{\lambda}} (-1)^\eta f_\eta
    \text{\;\; and \;\;} 
    f^k = \sideset{}{^*}\sum_{\eta < \lambda_k} (-1)^\eta f_\eta.
  $$
  Since we can extend the sequence $(f_\eta)_{\eta < \lambda_k}$ by 
  0 functions to a sequence in $\dusb_1$ of length $\omega^\lambda$, 
  using \eqref{e:tending to zero} we get that 
  $f^k \in \mathscr{B}_1^\lambda$. 
  By \eqref{e:sup=limit} we have $f^k \to f$, moreover, 
  \eqref{e:alternating sum general} for the sequence 
  $(f_\eta)_{\eta < \omega^\lambda + 1} \in \dusb_1$, where 
  $f_{\omega^{\lambda}} = g = \inf_{\eta < \omega^\lambda} f_\eta$ 
  is a USC function, yields 
  \begin{equation}
    \label{e:1}
    0 \le f - f^k \le f_{\lambda_k} - g.
  \end{equation}
  
  It is enough to prove that 
  $D_{f, \varepsilon}^{\omega^{\lambda + 1}}(X) = \emptyset$ 
  for every fixed $\varepsilon > 0$. In order to prove this let 
  $F_n = \{x \in X : g(x) \ge n \cdot \frac{\varepsilon}{12}\}$. 
  Note that $g$ is USC, hence $F_n$ is closed for every $n \in \N$. 
  Since $\bigcap_n F_n = \emptyset$, it is enough to prove that 
  \begin{equation}
    \label{e:D(F_n) <= F_(n+1)}
    D_{f, \varepsilon}^{\omega^\lambda}(F_n) \subseteq F_{n + 1},
  \end{equation}
  since then by induction on $n$ one can easily get that 
  $D_{f, \varepsilon}^{\omega^{\lambda} \cdot n}(X) \subseteq F_n$, hence 
  \begin{equation*}
    D_{f, \varepsilon}^{\omega^{\lambda + 1}}(X) = 
    \bigcap_{n \in \N} D_{f, \varepsilon}^{\omega^{\lambda} \cdot n}(X)
    \subseteq \bigcap_{n \in \N} F_n = \emptyset. 
  \end{equation*}    
  
  Let $x \in F_n \setminus F_{n + 1}$. 
  Since $f_{\lambda_k} \to g$, there exists a $k$ such that 
  \begin{equation}
    \label{e:3}
    f_{\lambda_k}(x) - g(x) \le \frac{\varepsilon}{12}.
  \end{equation}
  Since $f_{\lambda_k}$ is USC, there exists a neighborhood $U \ni x$ 
  such that 
  $f_{\lambda_k}(y) < f_{\lambda_k}(x) + \frac{\varepsilon}{12}$ 
  for every $y \in U$. Using that $x \in F_n \setminus F_{n + 1}$, 
  we have $g(x) - g(y) \le \frac{\varepsilon}{12}$ for every $y \in F_n$. 
  Using \eqref{e:1}, the last two inequalities and \eqref{e:3} we get 
  that for every $y \in U \cap F_n$,
  $$
    0 \le f(y) - f^k(y) \le f_{\lambda_k}(y) - g(y) \le 
    f_{\lambda_k}(x) + \frac{\varepsilon}{12} - g(x) + 
    \frac{\varepsilon}{12} \le \frac{\varepsilon}{4}.
  $$
  
  Again applying Lemma \ref{l:derivatives of functions} with $g = f^k$, 
  $F = F_n$ and $\eta = \omega^\lambda$, we get that 
  $D_{f, \varepsilon}^{\omega^\lambda}(F_n) \cap U \subseteq 
  D_{f^k, \frac{\varepsilon}{4}}^{\omega^\lambda}(F_n) \cap U = 
  \emptyset$, hence $x \not\in D_{f, \varepsilon}^{\omega^\lambda}(F_n)$. 
  Since $x \in F_n \setminus F_{n + 1}$ was arbitrary, we get 
  \eqref{e:D(F_n) <= F_(n+1)} as desired. 
  This finishes the proof of \eqref{e:not tending to zero} and also 
  the proof of the theorem. 
\end{proof}

Now we prove an analogue of the previous theorem for 
the Baire class $\xi$ case. 

\begin{theorem}
  \label{t:Baire xi alternating}
  Let $f$ be a bounded Baire class $\xi$ function. Then 
  $f \in \mathscr{B}_\xi^\lambda$ if and only if 
  $\len_\xi(f) \le \omega^\lambda$. 
\end{theorem}
\begin{remark}
  If one considers $\len_\xi(f)$ as the rank of the function $f$, 
  then the theorem says that this rank essentially coincides with 
  $\alpha_\xi^*$, $\beta_\xi^*$ and $\gamma_\xi^*$ on the bounded Baire 
  class $\xi$ functions. 
\end{remark}
\begin{proof}
  First we prove that if $f \in \mathscr{B}_\xi^\lambda$ then 
  $\len_\xi(f) \le \omega^\lambda$. 
  By Remark \ref{r:bounded classes} we have a topology 
  $\tau' \in T_{f, \xi}$ such that 
  $f \in \mathscr{B}_1^\lambda(\tau')$. 
  Using Theorem \ref{t:Baire 1 alternating}, there is a sequence 
  $(f_\eta)_{\eta < \omega^\lambda} \in \dusb_1(\tau')$ and 
  $c \in \R$ with 
  $$
    f = c + \sideset{}{^*}\sum_{\eta < \omega^\lambda} 
    (-1)^\eta f_\eta.
  $$
  The function $f_\eta$ is USC in $\tau'$ for each $\eta$, hence 
  $\{f_\eta < c\} \in \boldsymbol{\Sigma}^0_1(\tau')$, and since 
  $\tau' \in T_{f, \xi}$, we have $\boldsymbol{\Sigma}^0_1(\tau') \subseteq 
  \boldsymbol{\Sigma}^0_\xi(\tau)$, thus $f_\eta$ is a semi-Borel 
  class $\xi$ function with respect to the original topology $\tau$. 
  From this, one can easily conclude that 
  $(f_\eta)_{\eta < \omega^\lambda} \in \dusb_\xi(\tau)$ and 
  consequently $\len_\xi(f) \le \omega^\lambda$, 
  proving this part of the theorem.
  
  For the other direction, suppose that 
  $\len_\xi(f) \le \omega^\lambda$, and let 
  $$
    f = c + \sideset{}{^*}\sum_{\eta < \omega^\lambda} 
    (-1)^\eta f_\eta, 
  $$
  where $(f_\eta)_{\eta < \omega^\lambda} \in \dusb_\xi$. Since 
  $\{f_\eta < q\} \in \boldsymbol{\Sigma}^0_\xi$ for every 
  $q \in \Q$, it can be written as 
  $\{f_\eta < q\} = \bigcup_n F_n^{\eta, q}$, where 
  $F_n^{\eta, q} \in \bigcup_{\zeta < \xi} \boldsymbol{\Pi}^0_\zeta 
  \subseteq \boldsymbol{\Delta}^0_\xi$. Using Kuratowski's theorem 
  (see e.g.~\cite[22.18]{K}), there exists a Polish refinement 
  $\tau' \supseteq \tau$ 
  such that $F_n^{\eta, q} \in \boldsymbol{\Delta}^0_1(\tau')$ 
  for every $\eta$, $n$ and $q \in \Q$, and 
  $\tau' \subseteq \boldsymbol{\Sigma}^0_\xi(\tau)$. 
  
  Now $\{f_\eta < q\} \in \boldsymbol{\Sigma}^0_1(\tau')$ for every 
  $\eta$ and $q \in \Q$, hence $f_\eta$ is USC in $\tau'$, since 
  $\{f_\eta < c\} = \bigcup_{n} \{f_\eta < q_n\}$ is open, where 
  $q_n \in \Q$, $q_n \to c$, $q_n < c$. From this 
  $(f_\eta)_{\eta < \omega^\lambda} \in \dusb_1(\tau')$, hence with 
  the application of Theorem \ref{t:Baire 1 alternating} for the space 
  $(X, \tau')$, we get 
  $f \in \mathscr{B}_1^\lambda(\tau')$. 
  Note that $\tau' \in T_{f, \xi}$, hence Remark 
  \ref{r:bounded classes} yields 
  $f \in \mathscr{B}_\xi^\lambda(\tau)$, completing the proof. 
\end{proof}

\section{A way of generating the classes $\mathscr{B}_\xi^\lambda$ from lower 
classes}

Kechris and Louveau introduced the notion of pseudouniform convergence. 
\begin{definition}[\cite{KL}]
  A sequence $(f_n)_{n \in \N}$ of functions is 
  \emph{pseudouniformly convergent} if 
  $\gamma((f_n)_{n \in \N}) \le \omega$, 
  as defined in \eqref{e:gamma(f_n)}. 
\end{definition}
\begin{definition}
  If $\mathcal{F}$ is a class of bounded Baire class 1 functions then 
  let $\Phi(\mathcal{F})$ be the set of those bounded Baire class 1 
  functions that are the pseudouniform limit of a sequence of functions 
  from $\mathcal{F}$, i.e., 
  \begin{equation*}
  \begin{split}
    \Phi(\mathcal{F}) = \{&f \in \mathcal{B}_1 : \text{$f$ is bounded, }
    \\ &\exists (f_n)_{n \in \N} \in \mathcal{F}^\N \left(
    \gamma((f_n)_{n \in \N}) \le \omega 
    \text{ and $f_n \to f$ pointwise}\right)\}. 
  \end{split}
  \end{equation*}
\end{definition}

Now we define inductively the families $\Phi_\lambda$ of functions by 
$\Phi_0 = \mathscr{B}_1^1$ and for $0 < \lambda < \omega_1$, 
$$
  \Phi_\lambda = \Phi\left( \bigcup_{\eta < \lambda} \Phi_\eta
  \right).
$$

\begin{theorem}
  \label{t:Baire 1 generate}
  For every ordinal $\lambda < \omega_1$, we have 
  $\Phi_\lambda = \mathscr{B}_1^{\lambda + 1}$. 
\end{theorem}
\begin{remark}
  \label{r:Polish failure}
  This theorem is a nice analogue of the well-known theorem that 
  a function is of Baire class $\lambda$ if and only if it is 
  Borel-$(\lambda + 1)$ (see e.g.~\cite[24.3, 24.10]{K}). 
\end{remark}
\begin{remark}
  The authors of \cite{KL} defined $\Phi_\lambda$ for limit $\lambda$ as 
  the uniform limits of functions from the smaller classes, and 
  they proved that in this case 
  $\Phi_\lambda = \mathscr{B}_1^\lambda$ (with $\Phi_0 = \mathscr{B}^0_1$), 
  if the space is compact. 
  However, this is not the case for arbitrary Polish spaces. 
  We sketch the proof of this. 
  
  First, for every $\lambda < \omega_1$, one can easily construct a 
  countable closed set $F_\lambda \subseteq \R$ and a subset 
  $A_\lambda \subseteq F_\lambda$ such that 
  the $\alpha$ rank of $\chi_{A_\lambda}$ in the space $F_\lambda$ is 
  equal to $\lambda$. (Let $F_\lambda$ be a set with Cantor-Bendixson rank 
  $\lambda$ (see \cite[6.12]{K}). Then choose $A_\lambda$ such that 
  $A_\lambda$ and $F_\lambda \setminus A_\lambda$ are both ``dense'' in 
  $F_\lambda$, meaning that if $F_\lambda^\alpha \subseteq F_\lambda$ is the $\alpha$th 
  iterated Cantor-Bendixson derivative of $F_\lambda$ then the 
  closures of both $A_\lambda \cap F_\lambda^\alpha$ and 
  $F_\lambda^\alpha \setminus A_\lambda$ contain every limit point of $F_\lambda^\alpha$.) 
  This step will not work in compact spaces as the $\alpha$ 
  rank of a characteristic function on a compact space is always a 
  successor ordinal. 
  
  Then, it is easy to see that $\chi_{A_{\omega^\omega}}$ cannot be the 
  uniform limit of functions from 
  $\bigcup_{n < \omega} \mathscr{B}_1^n$, since if 
  $\|f - \chi_{A_{\omega^\omega}}\| \le 1/3$ then 
  $\alpha(f) \ge \alpha(\chi_{A_{\omega^\omega}}) = \omega^\omega$. 
\end{remark}
\begin{proof}[Proof of Theorem \ref{t:Baire 1 generate}]
  We prove the theorem by transfinite induction. For $\lambda = 0$ it 
  is exactly the definition of $\Phi_0$. 
  
  To prove that $\Phi_\lambda \subseteq \mathscr{B}_1^{\lambda + 1}$, 
  it is enough to show that 
  \begin{equation}
    \label{e:Phi(B^l) <= B^(l + 1)}
    \Phi(\mathscr{B}_1^{\lambda}) \subseteq 
    \mathscr{B}_1^{\lambda + 1},
  \end{equation}
  since for successor $\lambda$ it is 
  exactly 
  what is required, and for limit $\lambda$ we have 
  $$
    \Phi_\lambda = 
    \Phi\left( \bigcup_{\eta < \lambda} \Phi_\eta \right) =  
    \Phi\left( \bigcup_{\eta < \lambda} \mathscr{B}_1^{\eta + 1} 
    \right) 
    \subseteq \Phi(\mathscr{B}_1^\lambda).
  $$
  
  Let $(f_n)_{n \in \N}$ be a sequence from 
  $\mathscr{B}_1^{\lambda}$ converging pointwise to a bounded 
  function $f$. 
  \begin{claim}
    For every closed set $F$ and $\varepsilon > 0$, 
    $$
      D_{f, \varepsilon}^{\omega^{\lambda}}(F) \subseteq 
      D_{(f_n)_{n \in \N}, \frac{\varepsilon}{4}}(F).
    $$
  \end{claim}
  \begin{proof}
    Let $x \in F \setminus 
    D_{(f_n)_{n \in \N}, \frac{\varepsilon}{4}}(F)$, we need to show that 
    $x \not \in D_{f, \varepsilon}^{\omega^{\lambda}}(F)$. By the 
    definition of the derivative, there exists a neighborhood 
    $x \in U$ and $N \in \N$ such that for every $y \in F \cap U$ and 
    $n, m \ge N$ we have $|f_n(y) - f_m(y)| < \frac{\varepsilon}{4}$. 
    As $f_n(y) \to f(y)$ for every $y \in X$, we have 
    $|f_N(y) - f(y)| \le \frac{\varepsilon}{4}$ for every 
    $y \in F \cap U$. Applying Lemma \ref{l:derivatives of functions} 
    with $g = f_N$ and $\eta = \omega^{\lambda}$, we get 
    $$
      D_{f, \varepsilon}^{\omega^{\lambda}}(F) \cap U \subseteq
      D_{f_N, \frac{\varepsilon}{4}}^{\omega^{\lambda}}(F) \cap U = 
      \emptyset,
    $$
    since $f_N \in \mathscr{B}_1^{\lambda}$. 
    Hence $x \not \in D_{f, \varepsilon}^{\omega^{\lambda}}(F)$, 
    proving the claim.
  \end{proof}
  Now suppose moreover that $\gamma((f_n)_{n \in \N}) \le \omega$, 
  we need to show that $\beta(f) \le \omega^{\lambda + 1}$.
  Applying the claim repeatedly with 
  $F = D_{(f_n)_{n \in \N}, \frac{\varepsilon}{4}}^n(X)$, 
  by induction we get for each 
  $n \in \N$ that 
  $D_{f, \varepsilon}^{\omega^{\lambda} \cdot n}(X) \subseteq 
  D_{(f_n)_{n \in \N}, \frac{\varepsilon}{4}}^n(X)$. 
  Taking the intersection for each $n \in \N$, we get 
  $D_{f, \varepsilon}^{\omega^{\lambda + 1}}(X) \subseteq 
  D_{(f_n)_{n \in \N}, \frac{\varepsilon}{4}}^\omega(X) = 
  \emptyset$, hence $f \in \mathscr{B}_1^{\lambda + 1}$, showing 
  \eqref{e:Phi(B^l) <= B^(l + 1)} and thus finishing 
  the proof of $\Phi_\lambda \subseteq \mathscr{B}_1^{\lambda + 1}$. 
  
  Now we show the other direction, i.e., that 
  $\Phi_{\lambda} \supseteq \mathscr{B}_1^{\lambda + 1}$. 
  We do this by transfinite induction on $\lambda$. This is 
  obvious for $\lambda = 0$. For $\lambda > 0$, using the statement 
  for each $\eta < \lambda$, we have 
  $\Phi_{\lambda} = 
  \Phi\left(\bigcup_{\eta < \lambda} \Phi_\eta\right) =  
  \Phi\left(\bigcup_{\eta < \lambda} \mathscr{B}_1^{\eta + 1} \right)$, 
  hence it is enough to show that 
  $\Phi\left(\bigcup_{\eta < \lambda} \mathscr{B}_1^{\eta + 1} \right)
    \supseteq \mathscr{B}_1^{\lambda + 1}.$
  
  Let $f \in \mathscr{B}_1^{\lambda + 1}$ be a characteristic 
  function, i.e., $f = \chi_A$ for some $A \subseteq X$. 
  Using the same argument as in the proof of 
  Theorem \ref{t:Baire 1 alternating}, $A$ can be written as 
  $$
    A = \bigcup_{\begin{subarray}{c}
      \eta < \omega^{\lambda + 1} \\ \text{$\eta$ is even}
    \end{subarray}} F_\eta \setminus F_{\eta + 1},
  $$
  where $(F_\eta)_{\eta < \omega^{\lambda + 1}}$ is a decreasing, 
  continuous sequence of closed sets with $F_0 = X$ and 
  $\bigcap_{\eta < \omega^{\lambda + 1}} F_\eta = \emptyset$.
  
  Let $\lambda_k \to \omega^\lambda$, $\lambda_k < \omega^\lambda$ be 
  an increasing sequence of even ordinals with $\lambda_k > 0$ and let 
  $$
    B_k = \bigcup_{n \in \N} \;\; \bigcup_{\begin{subarray}{c}
      \omega^\lambda\cdot n \le 
      \eta < \omega^\lambda \cdot n + \lambda_k \\ \text{$\eta$ is even}
    \end{subarray}} F_\eta \setminus F_{\eta + 1}.
  $$
  Let $f_k = \chi_{B_k}$, it is easy to see that $f_k \to f$ 
  pointwise. We need to show that this convergence is pseudouniform, 
  and that $f_k \in \bigcup_{\eta < \lambda} \mathscr{B}_1^{\eta + 1}$ 
  for every $k \in \N$. 
  
  The proof of the former statement is based on the following claim. 
  \begin{claim}
    For every $n \in \N$ and $\varepsilon >0$ we have 
    $D_{(f_k)_{k \in \N}, \varepsilon}^n(X) \subseteq 
    F_{\omega^\lambda \cdot n}$.
  \end{claim}
  \begin{proof}
    For $n = 0$ this is the consequence of the definitions, so we 
    need to show that it holds for $n + 1$, if it holds for $n$. 
    For this, it is enough to show that 
    $D_{(f_k)_{k \in \N}, \varepsilon}(F_{\omega^\lambda \cdot n}) 
    \subseteq F_{\omega^\lambda \cdot (n + 1)}$. 
    Let $x \in F_{\omega^{\lambda} \cdot n} \setminus 
    F_{\omega^{\lambda} \cdot (n + 1)}$, we need to show that 
    $x \not \in D_{(f_k)_{k \in \N}, \varepsilon}
    (F_{\omega^{\lambda}\cdot n})$. 
    The sequence $(F_\eta)_{\eta < \omega^{\lambda + 1}}$ is 
    decreasing and continuous, hence 
    $F_{\omega^{\lambda} \cdot (n + 1)} = 
    \bigcap_{\eta < \omega^{\lambda} \cdot (n + 1)} F_\eta = 
    \bigcap_{k \in \N} 
    F_{\omega^{\lambda} \cdot n + \lambda_k}$, 
    so there is a $k \in \N$ such that $x \not \in 
    F_{\omega^{\lambda} \cdot n + \lambda_k}$. 
    
    Since $F_{\omega^{\lambda} \cdot n + \lambda_k}$ 
    is closed, there is a neighborhood $U \ni x$ such that 
    $U \cap F_{\omega^{\lambda} \cdot n + \lambda_k} 
    = \emptyset$. If $i, j \ge k$ then $f_i(y) = f_j(y)$ for all 
    $y \in U \cap F_{\omega^{\lambda} \cdot n}$, hence 
    $x \not \in D_{(f_k)_{k \in \N}, \varepsilon}
    (F_{\omega^{\lambda}\cdot n})$, proving the claim.
  \end{proof}
  Now 
  $$
    D_{(f_k)_{k \in \N}, \varepsilon}^\omega(X) = 
    \bigcap_{n \in \N} D_{(f_k)_{k \in \N}, \varepsilon}^n(X) 
    \subseteq \bigcap_{n \in \N} F_{\omega^\lambda\cdot n} = 
    \emptyset, 
  $$
  hence the convergence $f_k \to f$ is pseudouniform. 
  
  It remains to prove that 
  $f_k \in \bigcup_{\eta < \lambda} \mathscr{B}_1^{\eta + 1}$ for each 
  $k$. 
  
  \begin{claim}
    \label{c:derivatives}
    For every $\varepsilon > 0$ and $m \in \N$ we have 
    $D_{f_k, \varepsilon}^{(\lambda_k + 4) \cdot m}(X) 
    \subseteq F_{\omega^\lambda \cdot m}$.
  \end{claim}
  First we show that it is enough to prove the claim. 
  Since $\lambda_k > 0$, 
  $(\lambda_k + 4)\cdot \omega = \lambda_k \cdot \omega$, hence 
  using the fact that $\bigcap_{\eta < \omega^{\lambda + 1}} F_\eta = 
  \emptyset$ we have 
  $$
    D_{f_k, \varepsilon}^{\lambda_k \cdot \omega}(X) = 
    \bigcap_{m \in \N} 
    D_{f_k, \varepsilon}^{(\lambda_k + 4) \cdot m}(X) \subseteq 
    \bigcap_{m \in \N} F_{\omega^\lambda \cdot m} = \emptyset, 
  $$
  showing that $\beta(f_k) \le \lambda_k \cdot \omega$. 
  If $\lambda$ is limit then $\lambda_k \le \omega^\theta$ for some 
  $\theta < \lambda$, hence $\beta (f_k) \le \lambda_k \cdot \omega 
  \le \omega^{\theta + 1}$, showing that 
  $f_k \in \bigcup_{\eta < \lambda} \mathscr{B}_1^{\eta + 1}$ in this 
  case. 
  If $\lambda$ is successor then let $\lambda = \theta + 1$. 
  Now $\lambda_k < \omega^\theta \cdot l$ for some $l \in \N$, 
  hence $\lambda_k \cdot \omega \le \omega^{\theta + 1}$, showing 
  that $f_k \in \mathscr{B}_1^{\theta + 1} \subseteq 
  \bigcup_{\eta < \lambda} \mathscr{B}_1^{\eta + 1}$. 
  Now it only remains to prove the claim.
  
  \begin{proof}[Proof of Claim \ref{c:derivatives}]
    We prove this by induction on $m$. For $m = 0$ this is the 
    consequence of the definitions. Suppose it holds for $m$, to 
    prove it for $m + 1$ we need to show that if 
    $x \in F_{\omega^\lambda \cdot m} \setminus 
    F_{\omega^\lambda \cdot (m + 1)}$ then 
    $x \not \in D_{f_k, \varepsilon}^{\lambda_k + 4}
    (F_{\omega^\lambda \cdot m})$.
  
    There exists a neighborhood $U$ of $x$ with 
    $U \cap F_{\omega^\lambda \cdot (m + 1)} = \emptyset$ and let 
    $$
      H = \bigcup_{\begin{subarray}{c}
        \omega^\lambda \cdot m \le \eta < 
        \omega^\lambda \cdot m + \lambda_k \\ \text{$\eta$ is even}
        \end{subarray}} F_\eta \setminus F_{\eta + 1}.
    $$
    It is easy to see that 
    $\alpha_1(\chi_H) \le \lambda_k + 4$, since $H$ can be 
    written as the transfinite difference of closed sets of length 
    $\lambda_k + 4$ as the sequence 
    $(P_\eta)_{\eta < \lambda_k + 4}$, where 
    \begin{equation*}
      P_\eta = \left\{ 
      \begin{array}{cl}
        X & \text{if $\eta = 0$ or $1$} \\
        F_{\omega^\lambda \cdot m + \eta - 2} & 
          \text{if $2 \le \eta < \omega$} \\
        F_{\omega^\lambda \cdot m + \eta} & 
          \text{if $\omega \le \eta < \lambda_k$} \\
        F_{\omega^\lambda \cdot m + \lambda_k} & 
          \text{if $\lambda_k \le \eta \le \lambda_k + 1$} \\
        \emptyset & 
          \text{if $\lambda_k + 2 \le \eta \le \lambda_k + 3$}
      \end{array}\right.
    \end{equation*}  
    works. Note that $\lambda_k$ is even, hence $H$ is really the 
    transfinite difference of the sequence. 
    Using Lemma \ref{l:alpha = beta for char} and \cite[3.14]{EKV} we 
    have $\beta(\chi_H) = \alpha(\chi_H) \le \alpha_1(\chi_H) \le 
    \lambda_k + 4$.
    But as 
    $f_k(y) = \chi_{B_k}(y) = \chi_H(y)$ for every 
    $y \in F_{\omega^\lambda \cdot m} \cap U$, we have 
    $D_{f_k, \varepsilon}^{\lambda_k + 4}(F_{\omega^\lambda \cdot m})
    \cap U = D_{\chi_H, \varepsilon}^{\lambda_k + 4}
    (F_{\omega^\lambda \cdot m}) \cap U = \emptyset$, hence 
    $x \not \in D_{f_k, \varepsilon}^{\lambda_k + 4}
    (F_{\omega^\lambda \cdot m})$, proving the claim. 
  \end{proof}
  This finishes the proof that $f \in \Phi_\lambda$ for a 
  characteristic function $f \in \mathscr{B}_1^{\lambda + 1}$. 
  
  Now let $f \in \mathscr{B}_1^{\lambda + 1}$ be a step function, 
  i.e., $f = \sum_{i = 1}^n c_i \chi_{A_i}$, where the $c_i$'s are 
  distinct real numbers and the $A_i$'s form a partition of $X$. 
  For each $i$, 
  $\chi_{A_i} \in \mathscr{B}_1^{\lambda + 1}$ by \cite[3.38]{EKV}, 
  hence for each $i$ there exists a sequence $(f_i^k)_{k \in \N}$, 
  such that $(f_i^k)_{k \in \N} \to \chi_{A_i}$ pseudouniformly, 
  and $f_i^k \in \bigcup_{\eta < \lambda} \mathscr{B}_1^{\eta + 1}$. 
  Let $f^k = \sum_{i = 1}^n c_i \cdot f_i^k$. 
  Using Lemma \ref{l:gamma is additive}, 
  $\gamma((f^k)_{k \in \N}) \le \omega$, and it can be easily seen 
  that $f^k \to f$ pointwise. It remains to prove that 
  $f^k \in \bigcup_{\eta <\lambda} \mathscr{B}_1^{\eta + 1}$ for each 
  $k$. Let $k \in \N$ be fixed, then 
  $f_i^k \in \mathscr{B}_1^{\lambda_i + 1}$ for some 
  $\lambda_i < \lambda$. Hence with 
  $\lambda' = \max\{\lambda_i : 1 \le i \le n\} < \lambda$ we have 
  $f_i^k \in \mathscr{B}_1^{\lambda' + 1}$ for every $i$. 
  Now \cite[3.29]{EKV} yields that $f^k \in \mathscr{B}_1^{\lambda' + 1} 
  \subseteq \bigcup_{\eta < \lambda} \mathscr{B}_1^{\eta + 1}$, 
  proving that $f \in \Phi_\lambda$. 
    
  To finish the proof of the theorem, it remains to prove that 
  $f \in \Phi_\lambda$ for an arbitrary 
  $f \in \mathscr{B}_1^{\lambda + 1}$. 
  
  Let $f \in \mathscr{B}_1^{\lambda + 1}$. 
  By Lemma \ref{l:sum of step functions} there exists a sequence 
  $(g^k)_{k \in \N}$ of non-negative step-functions such that 
  $g^k \in \mathscr{B}_1^{\lambda + 1}$, $\inf f + \sum_k g^k = f$ 
  and $\|g^k\| \le \frac{1}{2^k}$ for $k \ge 1$. We can replace 
  $g^0$ with $g^0 + \inf f$, so now we have $\sum_k g^k = f$. 
  Since $g^k$ is a step-function, $g^k \in \Phi_\lambda$, hence for each $k$ 
  we have a sequence $(g_n^k)_{n \in \N}$ tending pseudouniformly 
  to $g^k$ with $g_n^k \in \bigcup_{\eta < \lambda} \Phi_\eta = 
  \bigcup_{\eta < \lambda} 
  \mathscr{B}_1^{\eta + 1}$ for each $n, k \in \N$. We first show that we 
  can suppose that $\|g_n^k\| \le \|g^k\|$. For every $k \in \N$ let 
  $h^k : \R \to \R$ be the following function: 
  \begin{equation*}
    h^k(x) = \left\{ 
    \begin{array}{cl}
      0 & \text{if $x < 0$}, \\
      x & \text{if $0 \le x \le \frac{1}{2^k}$}, \\
      \frac{1}{2^k} & \text{if $\frac{1}{2^k} < x$}.
    \end{array} \right. 
  \end{equation*}
  Then $h^k$ is a Lipschitz function, hence $\beta(h^k \circ g_n^k) \le 
  \beta(g_n^k)$ using Lemma \ref{l:beta(g f) <= beta(f)}, thus 
  $h^k \circ g_n^k \in \bigcup_{\eta < \lambda}\mathscr{B}_1^{\eta + 1}$. 
  Using the same arguments as in the proof of 
  Lemma \ref{l:beta(g f) <= beta(f)}, it is easy to see that 
  $\gamma((h^k \circ g_n^k)_{n \in \N}) \le \gamma((g_n^k)_{n \in \N}) \le 
  \omega$, hence the sequence $(h^k \circ g_n^k)_{n \in \N}$ is 
  pseudouniformly convergent for every $k$. Using the continuity of 
  $h^k$ we have $(h^k \circ g_n^k)_{n \in \N} \to h^k \circ g^k = g^k$. 
  This shows that by substituting $g_n^k$ with $h^k \circ g_n^k$, we can 
  really assume that $\|g_n^k\| \le \|g^k\|$.
  
  Now we prove the following claim. 
  \begin{claim}
    Let $f_n = \sum_{k \le n} g_n^k$, then the sequence 
    $(f_n)_{n \in \N}$  
    tends pseudouniformly to $f$.
  \end{claim}
  \begin{proof}
    First we show that $f_n \to f$ pointwise. Let $\varepsilon > 0$ 
    and $x \in X$ be fixed, and let $K \in \N$ be large enough so that 
    $\frac{1}{2^{K - 2}} < \frac{\varepsilon}{2}$. Then there exists a 
    common $N \ge K \in \N$ such that for all $k < K$ and $n > N$ we 
    have 
    $|g_n^k(x) - g^k(x)| \le \frac{\varepsilon}{2K}$.
    Thus, for $n > N$, 
    \begin{equation*}
    \begin{split}
      |f_n(x) - f(x)| = 
      \left|\sum_{k \le n} g_n^k(x) -
        \sum_{k \in \N} g^k(x) \right| \le \\
      \sum_{k < K}|g_n^k(x) - g^k(x)| + 
        \sum_{K \le k \le n} |g_n^k(x)| + 
        \sum_{k \ge K} |g^k(x)| \le 
      \frac{\varepsilon}{2K}\cdot K + 2\cdot \frac{1}{2^{K - 1}} \le 
      \varepsilon,
    \end{split}
    \end{equation*}
    proving the pointwise convergence. 

    Let $\varepsilon > 0$, it remains to show that 
    $D_{(f_n)_{n \in \N}, \varepsilon}^\omega(X) = \emptyset$. 
    Let $K \in \N$ be large enough so that 
    $2\frac{1}{2^K} < \frac{\varepsilon}{2}$. 
    Then for $n, m \ge K$ we have 
    $$
      \|f_n - f_m\| = 
      \left\| \sum_{k \le n} g_n^k - \sum_{k \le m} g_m^k \right\| \le 
      \left\| \sum_{k \le K} g_n^k - \sum_{k \le K} g_m^k \right\| + 
      2\frac{1}{2^K}, 
    $$
    hence if $|f_n(y) - f_m(y)| \ge \varepsilon$ then 
    $\left| \sum_{k \le K} g_n^k(y) - \sum_{k \le K} g_m^k(y)\right| 
    \ge \frac{\varepsilon}{2}$. 
    From this, using transfinite induction, one can 
    easily get for all $\eta < \omega_1$ that 
    $$
      D_{(f_n)_{n \in \N}, \varepsilon}^\eta(X) \subseteq 
      D_{(\sum_{k \le K} g_n^k)_{n \in \N}, \frac{\varepsilon}
      {2}}^\eta(X). 
    $$
    Using Lemma \ref{l:gamma is additive} the sequence 
    $(\sum_{k \le K} g_n^k)_{n \in \N}$ converges pseudouniformly 
    to $\sum_{k \le K} g^k$, hence 
    $D_{(\sum_{k \le K} g_n^k)_{n \in \N}, \frac{\varepsilon}
    {2}}^\omega(X) = \emptyset$, proving that 
    $D_{(f_n)_{n \in \N}, \varepsilon}^\omega(X) = \emptyset$.
  \end{proof}
  Using this claim it remains to prove that for each $n$, 
  $f_n \in \bigcup_{\eta < \lambda} \mathscr{B}_1^{\eta + 1}$. 
  Using the same idea as above, we have a $\lambda' < \lambda$ with 
  $g_n^k \in \mathscr{B}_1^{\lambda' + 1}$ for every $k \le n$, 
  hence by \cite[3.29]{EKV} we have 
  $f_n \in \mathscr{B}_1^{\lambda' + 1} \subseteq 
  \bigcup_{\eta < \lambda} \mathscr{B}_1^{\eta + 1}$.
  This show that $\Phi_\lambda \supseteq \mathscr{B}_1^{\lambda + 1}$, 
  finishing the proof of the theorem. 
\end{proof}

Now we give a generalized version of the above theorem for Baire 
class $\xi$ functions. From now on, let $1 < \xi < \omega_1$ 
be a fixed ordinal.

\begin{definition}
  Let $\mathcal{F}$ be a class of bounded Baire class $\xi$ functions 
  and let 
  \begin{equation*}
  \begin{split}
    \Phi(\mathcal{F}) = \Big\{f \in \mathcal{B}_\xi : 
    \text{$f$ is bounded, }
      \exists f_n \in \mathcal{F}, 
    \tau' \supseteq \tau \text{ Polish} \\
    \left(\tau' \subseteq \boldsymbol{\Sigma}^0_\xi(\tau), 
    f_n, f \in \mathcal{B}_1(\tau'), 
    f_n \to f \text{ pseudouniformly with respect to } \tau'\right)\Big\}.
  \end{split}
  \end{equation*}
  As in the Baire class 1 case, we define the families 
  $\Phi_\lambda$ as follows. Let $\Phi_0 = \mathscr{B}_\xi^1$ 
  and for $0 < \lambda < \omega_1$ let 
  $$
    \Phi_\lambda = 
      \Phi\left(\bigcup_{\eta < \lambda} \Phi_\eta\right). 
  $$
\end{definition}
  
\begin{theorem}
  For every ordinal $\lambda < \omega_1$, we have 
  $\Phi_\lambda = \mathscr{B}_\xi^{\lambda + 1}$. 
\end{theorem}
\begin{proof}
  For $\lambda = 0$ the statement is obvious. We first prove the 
  direction $\Phi_\lambda \supseteq \mathscr{B}_\xi^{\lambda + 1}$ 
  by transfinite induction on $\lambda$. 
  Let $f \in \mathscr{B}_\xi^{\lambda + 1}$. 
  By Remark \ref{r:bounded classes} there exists a 
  Polish topology $\tau' \supseteq \tau$ such that 
  $f \in \mathscr{B}_1^{\lambda + 1}(\tau')$. Thus, by 
  Theorem \ref{t:Baire 1 generate} there exists a sequence 
  $(f_n)_{n \in \N}$ of functions such that $f_n \to f$ 
  pseudouniformly in the topology $\tau'$, and for each $n$, 
  $f_n \in \bigcup_{\eta < \lambda} \mathscr{B}_1^{\eta + 1}(\tau')$. 

  It is easy to check from the definition that 
  $\tau' \in T_{f_n, \xi}$ for each $n$, hence Remark 
  \ref{r:bounded classes} now yields 
  $f_n \in \bigcup_{\eta < \lambda} \mathscr{B}_\xi^{\eta + 1}(\tau)$. 
  The sequence $(f_n)_{n \in \N}$ and the topology $\tau'$ is 
  exactly what is required by the above definition, showing that 
  $f \in \Phi\left(\bigcup_{\eta < \lambda} \mathscr{B}_\xi^{\eta + 1} 
  \right)$, proving $f \in \Phi_\lambda$. This proves that 
  $\Phi_\lambda \supseteq \mathscr{B}_\xi^{\lambda + 1}$. 
  
  We prove the other direction by transfinite induction on $\lambda$. 
  Let $f \in \Phi_\lambda$, i.e., there is a 
  sequence $(f_n)_{n \in \N}$ and a topology $\tau' \supseteq \tau$ 
  with $\tau' \subseteq \boldsymbol{\Sigma}^0_\xi(\tau)$, 
  $f, f_n \in \mathcal{B}_1(\tau')$,
  $f_n \to f$ pseudouniformly with respect to the topology $\tau'$ 
  and finally 
  $f_n \in \bigcup_{\eta < \lambda} \Phi_\eta = 
  \bigcup_{\eta < \lambda} \mathscr{B}_\xi^{\eta + 1}$, 
  using the induction hypothesis for each $\eta < \lambda$. 
  Consequently, there exists an ordinal $\lambda_n < \lambda$ for each 
  $n$, such that $f_n \in \mathscr{B}_\xi^{\lambda_n + 1}$. 
 
  Using Remark \ref{r:bounded classes} again, there exists a Polish 
  topology $\tau_n \in T_{f_n, \xi}$ such that 
  $f_n \in \mathscr{B}_1^{\lambda_n + 1}(\tau')$. 

  By \cite[5.12]{EKV} there exists a common Polish refinement $\tau''$ 
  of $\tau'$ and each $\tau_n$ with 
  $\tau'' \subseteq \boldsymbol{\Sigma}^0_\xi(\tau)$. 
  Then by \cite[5.13]{EKV} $f_n, f \in \mathcal{B}_1(\tau'')$, 
  moreover, $f_n \in \mathscr{B}_1^{\lambda_n + 1}(\tau')$ for each 
  $n$ and $\gamma_{\tau''}((f_n)_{n \in \N}) \le 
  \gamma_{\tau'}((f_n)_{n \in \N}) \le \omega$ can easily be seen 
  from the definition. 
  Theorem \ref{t:Baire 1 generate} yields that 
  $f \in \mathscr{B}_1^{\lambda + 1}(\tau'')$ but since 
  one can easily check that $\tau'' \in T_{f, \xi}$, we have 
  $f \in \mathscr{B}_\xi^{\lambda + 1}(\tau)$ again using 
  Remark \ref{r:bounded classes}, finishing the proof of the theorem. 
\end{proof}

\subsection*{Acknowledgment}
The author is greatly indebted to M.~Elekes for numerous 
suggestions about the manuscript.

The author was partially supported by the Hungarian Scientific Foundation 
grants no.~104178 and 113047.

\bigskip

\noindent Viktor Kiss\\
Department of Analysis\\
E\"otv\"os Lor\'and University\\
P\'az\-m\'any P. s. 1/c, H-1117, Budapest, Hungary\\
E-mail: kivi@cs.elte.hu
\end{document}